\documentclass{tran-l}

\newtheorem{theorem}{Theorem}[section]
\newtheorem{lemma}[theorem]{Lemma}
\newtheorem{proposition}[theorem]{Proposition}
\newtheorem{corollary}[theorem]{Corollary}

\theoremstyle{definition}
\newtheorem{definition}[theorem]{Definition}
\newtheorem{example}[theorem]{Example}

\theoremstyle{remark}
\newtheorem{remark}[theorem]{Remark}

\numberwithin{equation}{section}

\begin{document}

\title[FRACTIONAL IVPS]{Fractional integral equations\\ tell us how to impose initial values\\ in fractional differential equations}


%    author one information
\author[D. Cao Labora]{Daniel Cao Labora}
\address{Dept. of Statistics, Mathematical Analysis and Optimization,  University of Santiago de Compostela, Facultade de Matem\'aticas, Campus Vida, 
R\'ua Lope G\'omez de Marzoa s/n -- 15782 Santiago de Compostela, SPAIN }
\curraddr{}
\email{daniel.cao@usc.es}
\thanks{}

%    author two information
%\author{}
%\address{}
%\curraddr{}
%\email{}
%\thanks{}

%    \subjclass is required.


\date{October 6, 2019}

%\dedicatory{To my grandparents, Julio and Ricardo}

%    Abstract is required.
\begin{abstract}
The goal of this work is to discuss how should we impose initial values in fractional problems to ensure that they have exactly one smooth unique solution, where smooth simply means that the solution lies in a certain suitable space of fractional differentiability. For the sake of simplicity and to show the fundamental ideas behind our arguments, we will do this only for the Riemann-Liouville case of linear equations with constant coefficients.

In a few words, we study the natural consequences in fractional differential equations of the already existing results involving existence and uniqueness for their integral analogues, in terms of the Riemann-Liouville fractional integral. Under this scope, we derive naturally several interesting results. One of the most astonishing ones is that a fractional differential equation of order $\beta>0$ with Riemann-Liouville derivatives can demand, in principle, less initial values than $\lceil \beta \rceil$ to have a uniquely determined solution. In fact, if not all the involved derivatives have the same decimal part, the amount of conditions is given by $\lceil \beta - \beta_* \rceil$ where $\beta_*$ is the highest order in the differential equation such that $\beta-\beta_*$ is not an integer.
\end{abstract}

\maketitle

\section{Introduction}

One of the most typical trademarks involving Fractional Calculus is the wide range of opinions about the notions of what is a natural fractional version of some integer order concept and what is not. On the one hand, this plurality leads to very interesting debates and fosters a very relevant critical thinking about whether things are going ``in the right direction'' or not. On the other hand, it is difficult to handle such an amount of different notions and ideas in the extant literature, since there are usually lots of generalized fractional versions of a single integer order concept, some of them not very accurate. These debates are still very alive nowadays, and we are in a concrete moment where even the most fundamental aspects of fractional calculus are being reviewed, \cite{Lu}.

In this frame, the task of this paper is to point out some relevant facts concerning the imposition of initial values for Riemann-Liouville fractional differential equations, which is the most classical extension for the usual derivative, in the particular case of linear equations with constant coefficients. However, it seems natural that the ideas described here could be extended to much more general cases. 

In our opinion, we have to begin from the little things we are sure about. In this sense, if restrict the study of Fractional Calculus to functions defined on finite length intervals $[a,b]$, it is a big consensus that Riemann-Liouville fractional integral with base point $a$ is the unique reasonable extension for the integral operator $\int_a^t$. The previous asseveration is not a simple opinion, since Riemann-Liouville fractional integral can be characterized axiomatically in very reasonable terms. 

\begin{theorem}[Cartwright-McMullen, \cite{CaMc}]\label{TCaMc}
Given a fixed $a \in \mathbb{R}$, there is only one family of operators $\left(I_{a^+}^{\alpha}\right)_{\alpha > 0}$ on $L^1[a,b]$ satisfying the following conditions:
\begin{enumerate}
\item The operator of order $1$ is the usual integral with base point $a$. (Interpolation property)
\item The Index Law holds. That is, $I_{a^+}^{\alpha} \circ I_{a^+}^{\beta}=I_{a^+}^{\alpha+\beta}$ for all $\alpha,\beta > 0$. (Index Law)
\item The family is continuous with respect to the parameter. That is, the following map $\textnormal{Ind}_a:\mathbb{R}^+ \longrightarrow \textnormal{End}_B \left(L^1[a,b] \right)$ given by $\textnormal{Ind}_a(\alpha) = I_{a^+}^{\alpha}$ is continuous, where $\textnormal{End}_B \left(L^1[a,b] \right)$ denotes the Banach space of bounded linear endomorphisms on $L^1[0,b]$. (Continuity)
\end{enumerate}
This family is precisely given by the Riemann-Liouville fractional integrals, whose expression will be recalled during this paper.
\end{theorem}

Hence, it makes sense to study in detail fractional integral problems for the Riemann-Liouville fractional integral to derive consequences for the corresponding fractional equations afterwards. Finally, to draw the attention of curious readers, we mention again that one of the most interesting results that we have found out is that a fractional differential equation of order $\alpha>0$ with Riemann-Liouville derivatives can demand, in principle, less initial values than $\lceil \alpha \rceil$ to have a uniquely determined solution. A complete range of highlighted results with their implications can be consulted in Section \ref{s:conc}, while the previous sections are devoted to the corresponding deductions.

\subsection{Goal of the work}

The goal of this work is to study how should we impose initial values in fractional problems with Riemann-Liouville derivative to ensure that they have a smooth and unique solution, where smooth simply means that the solution lies in a certain suitable space of fractional differentiability. To achieve this, we will depart from the results involving the Riemann-Liouville fractional integral, since it arises as the natural generalization of the usual integral operator, recall Theorem \ref{TCaMc}.

First, we will recall some results that imply that fractional integral problems have always a unique solution. We also recall the fundamental notions concerning Fractional Calculus, and we pay special attention to the functional spaces where the performed calculations, and specially fractional derivatives, are well defined. Note that this point of ``where are functions defined'' is crucial to talk about existence or uniqueness of solution and is often neglected in the literature. Indeed, to avoid this problem, many research has been conducted for Caputo derivatives instead of Riemann-Liouville, see for instance \cite{DiFo, Maina} or general comments in \cite{MMKA}. The ideas of this paragraph are developed in the second section, and most of them are available in the extant literature, except (to the best of our knowledge) Lemma \ref{estructura}.

Second, we see how each fractional differential equation of order $\alpha$ is linked with a family of fractional integral problems, whose source term lives in a $\lceil \alpha \rceil$ dimensional affine subspace of $L^1[0,b]$. This means that each solution to the fractional differential equation is a solution to one (and only one!) fractional integral problem of the $\lceil \alpha \rceil$ dimensional family. Conversely, any solution to a fractional integral problem of the family is a solution to the fractional differential equation, provided that the solution is smooth enough. In general, the set of source terms of the family of fractional integral problems that provide a smooth solution will consist in an affine subspace of $L^1[0,b]$ of a dimension lower than $\lceil \alpha \rceil$. This is done in the third section.

Third, we characterize when a source term of the $\alpha$ dimensional family induces a smooth solution, and thus a solution for the associated fractional differential equation. This characterization induces a natural correspondence between each source term inducing a smooth solution for the integral problem and the vector of initial values fulfilled by the solution. This correspondence is performed in a way that ensures that the fractional differential problem has existence and uniqueness of solution. This final part is discussed in the fourth section.

Finally, we establish a section of conclusions to highlight the most relevant obtained results, and to point out to some relevant work that should be performed in the future to continue with this approach.

\section{Basic notions}

In this section we will introduce the basics notions of Fractional Calculus that we are going to use, together with their more relevant properties and some results of convolution theory that can be not so well known. We assume that the reader is familiar with the basic theory of Banach spaces, Special Functions and Integration Theory, specially the fundamental facts involving the space of integrable functions over a finite length interval, denoted by $L^1[a,b]$, and the main properties of the $\Gamma$ function.

\subsection{The Riemann-Liouville fractional integral}

We will briefly introduce the Riemann-Liouville fractional integral, together with its most relevant properties. We will make this introduction from the perspective of convolutions, since it will be relevant to notice that the Riemann-Liouville fractional integral is no more than a convolution operator, to apply later some adequate results of convolution theory.

\begin{definition} \label{Convolution} Given $f \in L^1[a,b]$, we defined its associated convolution operator as $C_a(f):L^1[a,b] \longrightarrow L^1[a,b]$ defined as $$(f *_a g)(t):=(C_a(f)\,g)(t):= \int_{a}^t f(t-s+a) \cdot g(s) \, ds$$ for $g \in L^1[a,b]$ and $t \in [a,b]$. Under the previous notation, we say that $f$ is the kernel of the convolution operator $C_a(f)$.
\end{definition}

\begin{definition} \label{DRiLiIn}
We define the left Riemann-Liouville fractional integral of order $\alpha>0$ of a function $f \in L^1[a,b]$ with base point $a$ as $$I_{a^+}^{\alpha}g(t)=\int_a^t \frac{(t-s)^{\alpha-1}}{\Gamma(\alpha)} \cdot g(s) \, ds,$$ for almost every $t \in [a,b]$. In the case that $\alpha=0$, we just define $$I_{a^+}^{0}\, g(t)=\textnormal{Id} \, g(t)=g(t).$$
\end{definition}

From now on we will assume that $a=0$, since the results for a generic value of $a$ can be achieved from the ones that we will mention or develop for $a=0$ after a suitable translation. Moreover, when using the expression ``Riemann-Liouville fractional integral'' we will understand that it is the left Riemann-Liouville fractional integral with base point $a=0$.

\begin{remark}
We observe that, for $\alpha>0$, the Riemann-Liouville fractional integral operator $I_{0^+}^{\alpha}$ can be written as a convolution operator $C_0(f)$, with kernel \[f(t)=\frac{t^{\alpha-1}}{\Gamma(\alpha)}.\]
\end{remark}

It is well known that the Riemann-Liouville fractional integral fulfils the following properties, see \cite{Samko}.

\begin{proposition}\label{propo} For every $\alpha,\beta \geq 0$:
\begin{itemize}
\item $I_{0^+}^{\alpha}$ is well defined, meaning that $I_{0^+}^{\alpha}L^1[0,b] \subset L^1[0,b]$.
\item $I_{0^+}^{\alpha}$ is a continuous operator (equivalently, a bounded operator) from the Banach space $L^1[0,b]$ to itself.
\item $I_{0^+}^{\alpha}$ is an injective operator.
\item $I_{0^+}^{\alpha}$ preserves continuity, meaning that $I_{0^+}^{\alpha}\mathcal{C}[0,b] \subset \mathcal{C}[0,b]$.
\item We have the Index Law $I_{0^+}^{\beta} \circ I_{0^+}^{\alpha} = I_{0^+}^{\alpha + \beta}$ for $\alpha,\beta \geq 0$. In particular, $I_{0^+}^{\alpha+\beta}L^1[0,b] \subset I_{0^+}^{\beta} L^1[0,b]$.
\item Given $f \in L^1[0,b]$ and $\alpha \geq 1$, we have that $I_{0^+}^{\alpha}f$ is absolutely continuous and, moreover, $I_{0^+}^{\alpha}f(0)=0$. 
\end{itemize}
\end{proposition}

Moreover, we will also use several times the following well known and straightforward remark, that can be obtained after direct computation and find in the basic bibliography involving Fractional Calculus \cite{KiSrTr, MiRo, Pod, Samko}.

\begin{remark}\label{Example}
We have that, for $\beta>-1$ and $\alpha \geq 0$, \[I_{0^+}^{\alpha}t^{\beta}=\frac{\Gamma(\beta+1)}{\Gamma(\alpha+\beta+1)}t^{\alpha+\beta} \in I_{0^+}^{\gamma}L^1[0,b].\] Indeed, $I_{0^+}^{\alpha}t^{\beta} \in I_{0^+}^{\gamma}L^1[0,b]$ if and only if $\alpha+\beta>\gamma-1$.
\end{remark}

\subsection{The Riemann-Liouville fractional derivative}

In this subsection, we will indicate the most relevant points when constructing the Riemann-Liouville fractional derivative. We will begin with a short introduction to absolutely continuous functions of order $n$, since the spaces where Riemann-Liouville differentiability is well defined can be understood as their natural generalization for the fractional case, see \cite{Samko}.

\subsubsection{Absolutely continuous functions and the Fundamental Theorem of Calculus}

We shall briefly indicate how the set of absolutely continuous functions is made up of the functions that, essentially, are antiderivatives of some function in $L^1[0,b]$. We will see later, how this notion is highly relevant to construct the spaces where fractional derivatives are well-defined.

\begin{definition} A real function $f$ of real variable is absolutely continuous on $[a,b]$ if for any $\varepsilon > 0$ there is $\delta >0$ such that for every family of subintervals $\{[a_1,b_1],...,[a_n,b_n]\}$ with disjoint interiors we have $$\sum_{k=1}^n (b_k-a_k) < \delta \Longrightarrow \sum_{k=1}^n \vert f(b_k)-f(a_k)\vert < \varepsilon.$$ We denote the set of this functions by $AC[a,b]$.
\end{definition}

\begin{remark} It follows trivially from the definition that any absolutely continuous function is uniformly continuous and, hence, continuous.
\end{remark}

This definition appears to carry little information. Furthermore, is seems a bit complicated to check the absolute continuity of a given function if its expression is not very manageable. However, the following theorem characterizes the absolutely continuous functions in a simple way.

\begin{theorem}[Fundamental Theorem of Calculus] \label{absolute} Consider a real function $f$ defined on an interval $[0,b] \subset \mathbb{R}$. Then, $f \in AC[0,b]$ if and only if there exists $\varphi \in L^1[0,b]$ such that
\begin{equation} \label{abscon}
f(t)=f(a)+\int_{0}^t \varphi(s) \, ds.
\end{equation}
\end{theorem}

\begin{remark} This result establishes that, essentially, absolutely continuous functions defined on $[0,b]$ are the primitives of the functions of $L^1[0,b]$, that is, antiderivatives of measurable functions whose absolute value has finite integral.
\end{remark}

This last result allows us to define the derivative of an absolutely continuous function on $[0,b]$ as a certain function in $L^1[0,b]$.

\begin{definition} \label{weak} If $f \in AC([0,b])$, we define its derivative $D^1 f$ as the unique function $\varphi \in L^{1}[0,b]$ that makes (\ref{abscon}) hold.
\end{definition}

\begin{remark} It is relevant to have in mind that the previous definition makes sense because, once fixed $f(0)$, the antiderivative operator $I_{0^+}^1$ is injective when defined on $L^1[0,b]$, recall Proposition \ref{propo}. In particular, \[AC[0,b]= \left\langle\{ 1 \}\right \rangle \oplus I_{0^+}^1 L^1[0,b],\] where ``$1$'' denotes the constant function with value $1$.
\end{remark}

\begin{definition} For any $n \in \mathbb{Z}^+$, we say that $f \in AC^{n}[0,b]$ provided that $f \in \mathcal{C}^{n-1}[0,b]$ and $D^{n-1}f \in AC[0,b]$.
\end{definition}

Thus, $AC^n[0,b]$ consists of functions that can be differentiated $n$ times, but the last derivative might be computable only in the weak sense of Definition \ref{weak}. Analogously to the previous remark, we have the following result, see page 3 of \cite{Samko}.

\begin{remark} \label{direct} We have that \[AC^{n}[0,b]= \left\langle\{ 1, t, \dots, t^{n-1}\}\right \rangle \oplus I_{0^+}^n L^1[0,b],\] after applying $n$ times the Fundamental Theorem of Calculus \ref{absolute}. Moreover, the sum is direct since the property $f \in I_{0^+}^n L^1[0,b]$ implies that $f(0)=f'(0)=\cdots=f^{n-1)}(0)=0$ and the only polynomial of degree at most $n-1$ satisfying such conditions is the zero one.
\end{remark}

The key observation is that the vector space of functions that can be differentiated $n$ times, in the sense of Fundamental Theorem of Calculus \ref{absolute}, has two disjoint parts that only share the zero function. The left one $\left\langle\{ 1, t, \dots, t^{n-1}\}\right \rangle$, which are polynomials of degree strictly lower than $n$, consists in the functions that are annihilated by the operator $D^n$ and, thus, \[\ker D^n=\left\langle\{ 1, t, \dots, t^{n-1}\}\right \rangle.\] The right part $I_{0^+}^n L^1[0,b]$ consists of functions that are obtained after integrating $n$ times an element of $L^1[0,b]$, and hence it contains functions of trivial initial values until the derivative of order $n-1$.

We have the following well known relation between the spaces $AC^{n}[0,b]$, with respect to the inclusion.

\begin{proposition} If $n > m> 0$, we have that \[AC^{n}[0,b] \subset AC^{m}[0,b].\]
\end{proposition}

\begin{proof} We note that \[AC^{n}[0,b]= \left\langle\{ 1, t, \dots, t^{m-1}\}\right\rangle \oplus \left\langle\{t^m, \dots, t^{n-1}\}\right\rangle \oplus I_{0^+}^n L^1[0,b]\] and we make the following straightforward claims
\begin{itemize}
\item $\left\langle\{ 1, t, \dots, t^{m-1}\}\right\rangle \subset AC^{m}[0,b]$,
\item $I_{0^+}^n L^1[0,b] \subset I_{0^+}^m L^1[0,b] \subset AC^{m}[0,b]$.
\end{itemize}
Therefore, we only need to check $\left\langle\{t^m, \dots, t^{n-1}\}\right\rangle \subset AC^{m}[0,b]$. We will only need to see that, indeed, \[\left\langle\{t^m, \dots, t^{n-1}\}\right\rangle \subset I_{0^+}^{m}L^1[0,b],\] but this is trivial since \[\left\langle\{t^m, \dots, t^{n-1}\}\right\rangle \subset I_{0^+}^{m}\left(\left\langle\{t^1, \dots, t^{n-m-1}\}\right\rangle\right).\]
\end{proof}

Although the previous result seems pretty immediate and irrelevant, it hides the key for a successful treatment of the fractional case. In the next part of the paper, we will reproduce the natural construction of the fractional analogue of the spaces $AC^n[0,b]$. For this construction, already presented in \cite{Samko}, it is not true that the space of order $\alpha$ is contained in the space of order $\beta$ if $\alpha > \beta$. Indeed, this particular behaviour will imply the existence of functions that can differentiated $\alpha$ times, but not $\beta$ times, which is surprising since we are assuming that $\alpha > \beta$.

\subsubsection{The fractional abstraction}

We define the Riemann-Liouville fractional derivative as the left inverse operator for the fractional integral. After that, an easy analytical expression for its computation, available in the classical literature, for instance \cite{Samko}, follows

\begin{definition} \label{inversas} Consider $\alpha \geq 0$. We define the Riemann-Liouville fractional derivative of order $\alpha$ (and base point $0$) as the left inverse of the corresponding Riemann-Liouville fractional integral, meaning
\begin{align*} D_{0^+}^{\alpha} I_{0^+}^{\alpha} f = f,
\end{align*}
for every $f \in L^1[0,b]$. 
\end{definition}

We should note that the Riemann-Liouville fractional derivative is well defined, due to the injectivity of the fractional integral, recall Proposition \ref{propo}. Moreover, it will be a surjective operator from $I_{0^+}^{\alpha}L^1[0,b]$ to $L^1[0,b]$. However, it is clear that we are missing something if we pretend that $D_{0^+}^{\alpha}$ matches perfectly the usual derivative when $\alpha$ is an integer. In particular, observe that for an integer value of $\alpha$, Definition \ref{inversas} only describes the behaviour of $D^{\alpha}$ over the space $I_{0^+}^{\alpha}L^1[0,b]$, but we are missing the behaviour over the complementary part in $AC^{\alpha}[0,b]$, which is $\ker D^{\alpha}$.

It happens that it is possible to describe Definition \ref{inversas} more explicitly, since the left inverse for $I_{0^+}^{\alpha}$ is clearly $D_{0^+}^{\alpha}=D_{0^+}^{\lceil \alpha \rceil}I_{0^+}^{\lceil \alpha \rceil - \alpha}$, due to the Fundamental Theorem of Calculus \ref{absolute} and Proposition \ref{propo}. Thus, one could define $D_{0^+}^{\alpha}$ in a more general space than $I_{0^+}^{\alpha}L^1[0,b]$, since the only necessary condition to define $D_{0^+}^{\lceil \alpha \rceil}I_{0^+}^{\lceil \alpha \rceil - \alpha}f$ is to ensure that $I_{0^+}^{\lceil \alpha \rceil - \alpha}f \in AC^{\lceil \alpha \rceil}[0,b]$. Hence, the following definition makes sense.

\begin{definition} \label{RLDerEsp} For each $\alpha >0$ we construct the following space $$\mathcal{X}_\alpha=\left(I_{0^+}^{\lceil \alpha \rceil -\alpha}\right)^{-1}\left(AC^{\lceil \alpha \rceil}[0,b]\right),$$ which will be called the space of functions with summable fractional derivative of order $\alpha$. If $\alpha=0$, we define $\mathcal{X}_\alpha=L^1[0,b]$.
\end{definition}

\begin{remark} Therefore, functions of $\mathcal{X}_{\alpha}$ are defined as the ones producing a function in $AC^{\lceil \alpha \rceil}[0,b]$ after being integrated $\lceil \alpha \rceil -\alpha$ times. This new function can be differentiated $\lceil \alpha \rceil$ times in the weak sense of Fundamental Theorem of Calculus \ref{absolute}.
\end{remark}

Now we see that this definition is, indeed, the same one that was already presented in \cite{Samko} with an explicit expression.

\begin{lemma} \label{equisa} For any $\alpha > 0$ we have that $$\mathcal{X}_\alpha=\left\langle\left\{ t^{\alpha-\lceil \alpha \rceil}, \dots , t^{\alpha-2}, t^{\alpha-1}\right\}\right\rangle \oplus I_{0^+}^{\alpha} L^1[0,b].$$
\end{lemma}

\begin{proof}
First, we check $\left\langle\left\{ t^{\alpha-\lceil \alpha \rceil}, \dots , t^{\alpha-2}, t^{\alpha-1}\right\}\right\rangle \cap I_{0^+}^{\alpha} L^1[0,b]=\{0\}$. If there is a function $f$ in both addends, then $I_{0^+}^{\lceil \alpha \rceil -\alpha} f$ will be simultaneously a polynomial of degree at most $\lceil \alpha \rceil-1$, and a function in $I_{0^+}^{\lceil \alpha \rceil} L^1[0,b]$. Therefore, $I_{0^+}^{\lceil \alpha \rceil -\alpha} f$ has to be the zero function after repeating the argument in Remark \ref{direct} and, since fractional integrals are injective (Proposition \ref{propo}), $f \equiv 0$.

It is clear that applying $I_{0^+}^{\lceil \alpha \rceil -\alpha}$ to the right hand side we will produce a function $AC^{\lceil \alpha \rceil}[0,b]$. Moreover, it is trivial that any function in $AC^{\lceil \alpha \rceil}[0,b]$ can be obtained in this way in virtue of Remark \ref{Example}. Since the operator $I_{0^+}^{\lceil \alpha \rceil -\alpha}$ is injective, the result follows.
\end{proof}

From the previous lemma, we get this immediate corollary.

\begin{corollary} \label{preimagen} Given $f \in L^1[0,b]$, we have that $f \in I_{0^+}^{\alpha}L^1[0,b]$ if, and only if, $f\in \mathcal{X}_{\alpha}$ and also $D^{s} I_{0^+}^{\lceil \alpha \rceil - \alpha}f(0)= 0$ for each $s \in \{0,\dots, \lceil \alpha \rceil-1\}$.
\end{corollary}

Hence, we can upgrade Definition \ref{inversas} in the following way, coinciding with Definition 2.4 in \cite{Samko}.

\begin{definition} \label{DerivadaFrac} Consider $\alpha \geq 0$ and $f \in \mathcal{X}_{\alpha}$. We define the Riemann-Liouville fractional derivative of order $\alpha$ (and base point $0$) as
\begin{align*}D_{0^+}^{\alpha} f :=  D_{0^+}^{\lceil \alpha \rceil} \circ I_{0^+}^{\lceil \alpha \rceil - \alpha} f,
\end{align*}
where the last derivative may be understood in the weak sense exposed previously.
\end{definition}

\subsubsection{Properties of the space $\mathcal{X}_{\alpha}$}

We want to fully understand how $D_{0^+}^{\alpha}$ works over $\mathcal{X}_{\alpha}$ and the most natural way is to split the problem into two parts, as suggested by Lemma \ref{equisa}. We already know that $D_{0^+}^{\alpha}$ is the left inverse for $I_{0^+}^{\alpha}$, so we should study how does it behave when applied to $\left\langle\left\{ t^{\alpha-\lceil \alpha \rceil}, \dots , t^{\alpha-2}, t^{\alpha-1}\right\}\right\rangle$. It is a well known and straightforward computation that \[D_{0^+}^{\alpha}\left(\left\langle\left\{ t^{\alpha-\lceil \alpha \rceil}, \dots , t^{\alpha-2}, t^{\alpha-1}\right\}\right\rangle\right)=\{0\}.\] and, hence the kernel of $D_{0^+}^{\alpha}$ has dimension $\lceil \alpha \rceil$ and is given by \[\ker D_{0^+}^{\alpha}=\left\langle\left\{ t^{\alpha-\lceil \alpha \rceil}, \dots , t^{\alpha-2}, t^{\alpha-1}\right\}\right\rangle.\]

Moreover, we should note that if $f(t)=a_0 \, t^{\alpha-\lceil \alpha \rceil}+\cdots+a_{\lceil \alpha \rceil-1}\,t^{\alpha -1}$, with $a_j\in \mathbb{R}$ for each $j \in \left\{0,1,\dots,\lceil \alpha \rceil-1\right\}$, it is immediate to do the following calculations from Remark \ref{Example}, where $j \in \{1,\cdots,\lceil \alpha \rceil -1\}$,

\begin{align} \label{Iniciales}
\begin{split}
\left( I_{0^+}^{\lceil \alpha \rceil - \alpha}f \right)(0)=a_0\,{\Gamma(\alpha-\lceil \alpha \rceil+1)},\\
\left( D_{0^+}^{\alpha - \lceil \alpha \rceil + j}f \right)(0)=a_j\,\Gamma(\alpha - \lceil \alpha \rceil + j+1).
\end{split}
\end{align}
The previous formula generalizes the obtention of the Taylor coefficients for a fractional case and it can be used to codify functions in $\mathcal{X}_{\alpha}$ modulo $I_{0^+}^{\alpha}L^1[0,b]$, since 
\begin{align} \label{Iniciales2}
\begin{split}
\left( I_{0^+}^{\lceil \alpha \rceil - \alpha}g \right)(0)=0,\\
\left( D_{0^+}^{\alpha - \lceil \alpha \rceil + j}g \right)(0)=0,
\end{split}
\end{align}
for $g \in I_{0^+}^{\alpha}L^1[0,b]$, due to Proposition \ref{propo}.

\subsubsection{Intersection of fractional summable spaces}

In general, fractional differentiation presents some extra problems that do not exist when dealing with fractional integrals. One of the most famous ones is that there is no Index Law for fractional differentiation. The main reason underlying all these complications is the following one.

\begin{remark} \label{nocont} The condition $\alpha > \beta$ does not ensure $\mathcal{X}_{\alpha} \subset \mathcal{X}_{\beta}$, although the condition $\alpha - \beta \in \mathbb{Z}^+$ trivially does. This makes Riemann-Liouville derivatives somehow tricky, since the differentiability for a higher order does not imply, necessarily, the differentiability for a lower order with different decimal part. In particular, this fact has critical implications when considering fractional differential equations, as we shall see in the paper, since the function has to be differentiable for each order involved in the equation. These problems give an idea of why can be a logical thought to work with fractional integrals instead, and try to inherit the results obtained for the case of fractional derivatives; instead of proving them for fractional derivatives directly.
\end{remark}

Consequently, it is interesting to compute the exact structure of a finite intersection of such spaces of different orders. To the best of our knowledge, this result is not available in the extant literature.

\begin{lemma}\label{estructura}
Consider $\beta_n > \dots > \beta_1 \geq 0$, we have that $$ \bigcap_{j=1}^n \mathcal{X}_{\beta_j} =\left\langle\left\{ t^{\beta_n-\lceil \beta_n-\beta_* \rceil}, \dots, t^{\beta_n -1} \right\}\right\rangle \oplus I_{0^+}^{\beta_n} L^1[0,b],$$ where $\beta_*$ is the maximum $\beta_j$ such that $\beta_n - \beta_j \not \in \mathbb{Z}^+$. If such a $\beta_j$ does not exist, the result still holds after defining $\beta_*=0$.

In particular, $I_{0^+}^{\beta_n} L^1[0,b] \subset \bigcap_{j=1}^n \mathcal{X}_{\beta_j}$ and it has codimension $\lceil \beta_n-\beta_* \rceil$.
\end{lemma}

\begin{proof}
It is obvious that $\bigcap_{j=1}^n \mathcal{X}_{\beta_j} \subset \mathcal{X}_{\beta_n}$. Hence,
\begin{equation}\label{ecurara}
\bigcap_{j=1}^n \mathcal{X}_{\beta_j} \subset \mathcal{X}_{\beta_n} = \left\langle \left\{ t^{\beta_n-\lceil \beta_n \rceil}, \dots, t^{\beta_n -1} \right\}\right\rangle \oplus I_{0^+}^{\beta_n} L^1[0,b].
\end{equation}
It is clear that $I_{0^+}^{\beta_n} L^1[0,b]$ lies in $\bigcap_{j=1}^n \mathcal{X}_{\beta_j}$, so the remaining question is to see when a linear combination of the $t^{\beta_n-k}$, where $k \in \{1,\dots,\lceil \beta_n \rceil \}$, lies in $\bigcap_{j=1}^n \mathcal{X}_{\beta_j}$.

The key remark is to realise that for any finite set $\mathcal{F} \subset (-1,+\infty)$ $$\sum_{\gamma \in \mathcal{F}} c_{\gamma}\, t^{\gamma} \in \mathcal{X}_{\beta_j}, \textnormal{ where } c_{\gamma} \neq 0 \textnormal{ for each } \gamma \in \mathcal{F},$$ if and only if $\gamma-\beta_j>-1$ or $\gamma-\beta_j \in \mathbb{Z}^{-}$ for every $\gamma \in \{1,\dots,r\}$ with $c_{\gamma}\neq 0$. Consequently, it is enough study when $t^{\beta_n-k}$ lies in $\mathcal{X}_{\beta_j}$, and there are two options:
\begin{itemize}
\item If $\beta_n-\beta_j \in \mathbb{Z}^+$, we know that $t^{\beta_n-k} \in \mathcal{X}_{\beta_j}$. This happens because either $\beta_n -k \in \left\{\beta_j-\lceil \beta_j \rceil, \dots, \beta_j -1 \right\}$ or $\beta_n-k>\beta_j-1$.
\item In other case, we need $\beta_n-k>\beta_j-1$, that can be rewritten as $k<\beta_n-\beta_j+1$. If we want this to happen for every $j$ such that $\beta_n-\beta_j \not \in \mathbb{Z}$, the condition is equivalent to $k<\beta_n-\beta_*+1$, where $\beta_*$ is the greatest $\beta_j$ such that $\beta_n-\beta_j \not \in \mathbb{Z}$. Indeed, it can be rewritten as $1 \leq k \leq \lceil \beta_n-\beta_* \rceil$.
\end{itemize}
Therefore, the coefficients which are not necessarily null are the ones associated to $t^{\beta_n-k}$, where $k \in \{1,\dots,\lceil \beta_n-\beta_* \rceil\}$.
\end{proof}

\begin{remark} \label{interse} Due to Lemma \ref{estructura}, any affine subspace of $ \bigcap_{j=1}^n \mathcal{X}_{\beta_j}$ with dimension strictly higher than $\lceil \beta_n - \beta_* \rceil$ contains two distinct functions whose difference lies in $I_{0^+}^{\beta_n}L^1[0,b]$. Thus, in any vector subspace of $ \bigcap_{j=1}^n \mathcal{X}_{\beta_j}$ with dimension strictly higher than $\lceil \beta_n - \beta_* \rceil$, there are infinitely many functions that lie in $I_{0^+}^{\beta_n}L^1[0,b]$.
\end{remark}

\subsection{Fractional integral equations}

Consider the fractional integral equation \[\left(c_n\, I_{0^+}^{\gamma_n}+\cdots+c_1\,I_{0^+}^{\gamma_1}+I_{0^+}^{\gamma_0}\right)x(t)=\widetilde{f}(t),\] where $f \in L^1[0,b]$, $\gamma_n > \cdots > \gamma_0 \geq 0$ and assume that it has a solution $x \in L^1[0,b]$. Since $I_{0^+}^{\gamma_n} L^1[0,b] \subset \cdots \subset I_{0^+}^{\gamma_0} L^1[0,b]$, the left hand side lies in $I_{0^+}^{\gamma_0} L^1[0,b]$ and the condition $f \in I_{0^+}^{\gamma_0} L^1[0,b]$ is mandatory to ensure the existence of solution. In that case, we can apply the operator $D_{0^+}^{\gamma_0}$ to the previous equation and we obtain 
\begin{align}\label{fundamental}
\left(c_n\,I_{0^+}^{\alpha_n}+\cdots+c_1\,I_{0^+}^{\alpha_1}+\textnormal{Id}\right)x(t)=f(t),
\end{align} where $D_{0^+}^{\gamma_0}\widetilde{f}=f$ and $\alpha_j=\gamma_j-\gamma_0$ for $j \in \{1,\dots,n\}$. If we use the notation \[\Upsilon=c_n\, I_{0^+}^{\alpha_n}+\cdots+c_1\,I_{0^+}^{\alpha_1},\] Equation (\ref{fundamental}) can be rewritten as
\begin{align}\label{fundamental2}
\left(\Upsilon+\textnormal{Id}\right)x(t)=f(t).
\end{align}

Therefore, it is relevant to study the properties of the operator $\Upsilon+\textnormal{Id}$ from $L^1[0,b]$ to itself to understand Equation (\ref{fundamental2}).

\subsubsection{$\Upsilon+\textnormal{Id}$ is bounded}

This claim is a very well-known result, since each addend in $\Upsilon$ is a bounded operator, and Id too, see \cite{KiSrTr}. It is also possible to prove this, just recalling that $\Upsilon$ is a convolution operator with kernel in $L^1[0,b]$ and, thus, a bounded operator.

\subsubsection{$\Upsilon+\textnormal{Id}$ is injective}

To prove that $\Upsilon+\textnormal{Id}$ is injective we will need a result concerning the annulation of a convolution. In a few words, we need to know what are the possibilities for the factors of a convolution, provided that the obtained result is the zero function. Roughly speaking, the classical result in this direction, known as Titchmarsh Theorem, states that the integrand of the convolution from $0$ to $t$ is always zero, independently of $t$.

\begin{theorem}[Titchmarsh, \cite{Titchmarsh}] \label{Titchmarsh}
Suppose that $f,g \in L^1[0,b]$ are such that $f*_0 g \equiv 0$. Then, there exist $\lambda,\mu \in \mathbb{R}^+$ such that the following three conditions hold:
\begin{itemize}
\item $f \equiv 0$ in the interval $[0,\lambda]$,
\item $g \equiv 0$ in the interval $[0,\mu]$,
\item $\lambda+\mu \geq b$.
\end{itemize}
\end{theorem}

We will not provide the proof of this result, since it is a bit technical and it is not interesting from the point of view of the work that we are going to develop. The proof can be consulted in \cite{Titchmarsh}.

\begin{remark} \label{RTi} In particular, Titchmarsh Theorem states that the operator $C_0(f):L^1[0,b] \longrightarrow L^1[0,b]$ is injective, provided that $f \in L^1[0,b]$ and that $f$ is not null at any interval $[0,\lambda]$ for $\lambda >0$.
\end{remark}

\begin{corollary} The operator $\Upsilon+\textnormal{Id}$ described in (\ref{fundamental2}) is injective.
\end{corollary}

\begin{proof} Note that we can not apply Theorem \ref{Titchmarsh} directly to $\Upsilon+\textnormal{Id}$, since it is not a convolution operator due to the ``Id'' term. However, $I_{0^+}^1 \circ \left(T+\textnormal{Id}\right)$ is a convolution operator and we conclude, following Remark \ref{RTi}, that $I_{0^+}^1 \circ \left(\Upsilon+\textnormal{Id}\right)$ is injective. If the previous composition is injective the right factor $\left(\Upsilon+\textnormal{Id}\right)$ has to be injective.
\end{proof}

\subsubsection{$\Upsilon+\textnormal{Id}$ is surjective}

In this case, we will use the following result, concerning Volterra integral equations of the second kind. This result essentially states that some family of integral equations do always have a continuous solution, provided that the source term is continuous.

\begin{theorem}[Rust] \label{Rust} Given $k \in L^1[0,b]$, the Volterra integral equation $$\left(C_0(k)\,v\right)(t)+v(t):=\int_0^t k(t-s) \cdot  v(s) \, ds + v(t) = w(t)$$ has exactly one continuous solution $v \in \mathcal{C}[0,b]$, provided that $w \in \mathcal{C}[0,b]$ and that the following two conditions hold:
\begin{itemize}
\item If $h \in \mathcal{C}[0,b]$, then $C_a(k)\,h \in \mathcal{C}[0,b]$.
\item If $n \in \mathbb{Z}^+$ is big enough, then $\left(C_0(k)\right)^n = C_0(\widetilde{k})$ for some $\widetilde{k} \in \mathcal{C}[0,b]$, 
\end{itemize}
\end{theorem}

We will not provide the proof, analogously to what we previously did with Titchmarsh Theorem. The result can be found in a more general context in \cite{Rust}, since here we have stated it for the particular case of convolution kernels.

\begin{remark} \label{Rru} We know that the image of $\Upsilon+\textnormal{Id}$ will lie in $\mathcal{C}[0,b],$ since fractional integrals map continuous functions into continuous functions (Proposition \ref{propo}) and $\Upsilon^n$ will be defined by a continuous kernel when $n\geq\alpha_1^{-1}$, which is the inverse of the least integral order in $\Upsilon$.
\end{remark}

We need to conclude that, indeed, the image of $\Upsilon+\textnormal{Id}$ is $L^1[0,b]$.

\begin{corollary} The operator $\Upsilon+\textnormal{Id}$ described in (\ref{fundamental2}) is surjective.
\end{corollary}

\begin{proof} Consider $f \in L^1[0,b]$ and the equation \[\left(\Upsilon+\textnormal{Id}\right)x(t)=f(t).\] Observe that $x$ solves the previous equation if and only if it solves \[\left(\Upsilon+\textnormal{Id}\right)(x(t)-f(t))=-\Upsilon\,f(t),\] but now the source term is in $I_{0^+}^{\alpha_1}L^1[0,b]$. If we repeat this idea inductively, we see that $x$ solves the original equation if and only if \[\left(\Upsilon+\textnormal{Id}\right)(x(t)-(\textnormal{Id}-\Upsilon+\cdots+(-1)^n \Upsilon)f(t))=(-1)^{n+1}\Upsilon^{n+1}\,f(t).\] The right hand side will be continuous for $n \geq \alpha_1^{-1}$ and, by Remark \ref{Rru}, it will have a solution.
\end{proof}

\subsubsection{$\Upsilon+\textnormal{Id}$ is a bounded automorphism in $L^1[0,b]$}

We have already seen that $\Upsilon+\textnormal{Id}$ is bounded and bijective, and hence the inverse is also bounded due to the Bounded Inverse Theorem for Banach spaces. Therefore, we have the following result.

\begin{theorem} The operator $T+\textnormal{Id}$, described in (\ref{fundamental2}) is an invertible bounded linear map from the Banach space $L^1[0,b]$ to itself, whose inverse is also bounded.
\end{theorem}

In particular, we get the following corollary

\begin{theorem} \label{uncidad} Given $f \in L^1[0,b]$, the equation
\begin{align}
\left(\Upsilon+\textnormal{Id}\right)x(t)=f(t) \tag{\ref{fundamental2}}
\end{align}
has exactly one solution $x \in L^1[0,b]$.
\end{theorem}

Although it is not the scope of this paper, we highlight that such an equation can be solved using classical techniques for integral equations or specifical tools for the particular case of fractional integral equations, like the one exposed in \cite{CaRo}.

\section{Implications of fractional integral equations in fractional differential equations}

It would be nice to inherit some of the previous results for fractional differential equations. In fact, we are interested in studying the solutions of this general linear problem with constant coefficients
\begin{equation} \label{previodiff}
L \, u(t):=\left(c_1 \, D_{0^+}^{\beta_1}+\dots+ c_{n-1}\,D_{0^+}^{\beta_{n-1}}+D_{0^+}^{\beta_n}\right) u(t) = w(t),
\end{equation}
where $\beta_n>\dots>\beta_1\geq 0$ and $w \in L^1[0,b]$. Of course, the first question is where should we look for the solution. It is very relevant to clarify completely this point, since there are classical references, for instance see \cite{MiRo} (Theorem 1, Section 5.5), that state the following theorem or equivalent versions.

\begin{theorem} \label{incompl} Consider a linear homogeneous fractional differential equation (for Riemann-Liouville derivatives) with constant coefficients and rational orders. If the highest order of differentiation is $\alpha$, then the equation has $\lceil \alpha \rceil$ linearly independent solutions.
\end{theorem}

It is important to note that many references for are not clear enough about the notion of solution to a fractional differential equation. With the previous sentence, we mean that it is desirable to introduce a suitable space of differentiable functions first, to later discuss about the solvability of the fractional differential equation. We devote the rest of the paper to show that the previous theorem is true only in some weak sense. Indeed, after defining formally the notion of ``strong solution'', we will see that, in general, there are less than $\lceil \alpha \rceil$ linearly independent solutions. Indeed, only for those ``strong'' solutions it will be coherent to talk about initial values.

If we go back to \eqref{previodiff}, we can make the following vital remark.

\begin{remark} We recall that, in the usual case of integer orders, we look for the solutions in $\mathcal{X}_{\beta_n}$. Although it is quite common to forget it, the underlying reason to do this is that $\bigcap_{j=1}^n \mathcal{X}_{\beta_j}=\mathcal{X}_{\beta_n}$ when every $\beta_j$ is a non-negative integer. This means that any function with summable derivative of order $\beta_n$ has summable derivative of any fewer order too. However, in general, this does not necessarily happen when the involved orders are non-integers. Thus, we may have $\bigcap_{j=1}^n \mathcal{X}_{\beta_j} \neq \mathcal{X}_{\beta_n}$ and, of course, a solution to Equation (\ref{previodiff}) has to lie in $\bigcap_{j=1}^n \mathcal{X}_{\beta_j}$.
\end{remark}

Consequently, it is convenient to know the structure of the set $\bigcap_{j=1}^n \mathcal{X}_{\beta_j}$, which has already been described in Lemma \ref{estructura}, to study existence and uniqueness of solution. Of course, to expect uniqueness of solution, some initial conditions have to be added to Equation (\ref{previodiff}), but this will be detailed in the next section. The fundamental remark is that Equation (\ref{previodiff}) can be rewritten as
\begin{equation} \label{previodiff2}
L \, u(t):=D_{0^+}^{\beta_n} \, \left(c_1 I_{0^+}^{\beta_n-\beta_1}+\dots+c_{n-1} I_{0^+}^{\beta_n-\beta_{n-1}}+\textnormal{Id}\right) u(t) = w(t).
\end{equation}
In consequence, it is quite natural to make the following reflection. If $u(t)$ solves (\ref{previodiff2}), it is because
\begin{equation} \label{weaksol}
\left(c_1 I_{0^+}^{\beta_n-\beta_1}+\dots+c_{n-1} I_{0^+}^{\beta_n-\beta_{n-1}}+ \textnormal{Id}\right) u(t) \in I_{0^+}^{\beta_n} w(t) + \ker D_{0^+}^{\beta_n}.
\end{equation}
We will refer to the set of solutions to Equation (\ref{weaksol}), as the set of weak solutions. The previous terminology obeys the following reason: although a solution to (\ref{previodiff2}) solves (\ref{weaksol}), the converse does not hold in general. The point is that a solution of (\ref{weaksol}) may not lie in $\bigcap_{j=1}^n \mathcal{X}_{\beta_j}$. Of course, if the weak solution lies in $\bigcap_{j=1}^n \mathcal{X}_{\beta_j}$, then it solves (\ref{previodiff2}). The set of solutions to (\ref{previodiff2}) will be called set of strong solutions.

At this moment, we know two vital things:

\begin{itemize}
\item We have already described $\ker D_{0^+}^{\beta_n}=\left \langle \left\{ t^{\beta_n-1} , \dots, t^{\beta_n-\lceil \beta_n \rceil} \right\}\right \rangle$, that is a vector space of dimension $\lceil \beta_n \rceil$. Therefore, the set of weak solutions has dimension $\lceil \beta_n \rceil$ too, since it is the image of the affine space $I_{0^+}^{\beta_n} w(t) + \ker D_{0^+}^{\beta_n}$ via the automorphism $T^{-1} \in \textnormal{Aut}_B(L^1[0,b])$.
\item The dimension of the set of strong solutions is bounded from above by $\lceil \beta_n - \beta_*\rceil$. If the dimension were higher we could find two different solutions to (\ref{previodiff2}) whose difference would lie in $I_{0^+}^{\beta_n}L^1[0,b]$, due to Remark \ref{interse}. After writing their difference as $I_{0^+}^{\beta_n} g$ with $g \neq 0$, it would trivially fulfil
\begin{equation*}
\left(c_1 I_{0^+}^{\beta_n-\beta_1}+\dots+c_{n-1} I_{0^+}^{\beta_n-\beta_{n-1}}+\textnormal{Id}\right)g(t) = 0,
\end{equation*}
which is not possible since the linear operator in the left hand side is injective.
\end{itemize}

From these remarks, there are some remaining points that need to be studied in detail. First, we prove that the bound $\lceil \beta_n - \beta_*\rceil$ is sharp by inspecting which of elements in $\ker D_{0^+}^{\beta_n}$ guarantee that the weak solution associated to those elements is, indeed, a strong one.

\begin{remark} \label{remarklenders} We have $\bigcap_{j=1}^n \mathcal{X}_{\beta_j} \subset I_{0^+}^{\beta_n-\lceil \beta_n - \beta_* \rceil+1-\varepsilon} L^1[0,b]$ for every increment $\varepsilon>0$, but not for $\varepsilon=0$. Moreover, note that if $f \in \ker D_{0^+}^{\beta_n}$ is chosen as the right addend in the right hand side in (\ref{weaksol}), we have that, for $\gamma \leq \beta_n$, $f \in I_{0^+}^{\gamma}L^1[0,b]$ if and only if $u \in I_{0^+}^{\gamma}L^1[0,b]$. Therefore, to have a strong solution, it is mandatory to select $f \in \left \langle \left\{t^{\beta_n-1} , \dots, t^{\beta_n-\lceil \beta_n \rceil} \right\} \right \rangle$.
\end{remark}

\begin{lemma}
If $u \in L^1[0,b]$ solves
\begin{equation}\label{cansado}
\left(c_1 I_{0^+}^{\beta_n-\beta_1}+\dots+c_{n-1} I_{0^+}^{\beta_n-\beta_{n-1}}+ \textnormal{Id}\right) u(t) = I_{0^+}^{\beta_n} w(t) + f(t)
\end{equation}
for $f \in \left \langle \left\{t^{\beta_n-1} , \dots, t^{\beta_n-\lceil \beta_n-\beta_* \rceil} \right\} \right \rangle \subset \ker D_{0^+}^{\beta_n}$, then $u \in \bigcap_{j=1}^n \mathcal{X}_{\beta_j}$.
\end{lemma}

\begin{proof}
If we use the notation $\Upsilon:=c_1 I_{0^+}^{\beta_n-\beta_1}+\dots+c_{n-1} I_{0^+}^{\beta_n-\beta_{n-1}}$, we deduce from Equation (\ref{cansado}) that $$\left(\Upsilon+ \textnormal{Id}\right) \left(u(t)-f(t)\right) = I_{0^+}^{\beta_n} w(t) - \Upsilon\, f(t).$$ Observe now that the addend $-\Upsilon\, f(t)$ can be decomposed in two parts, since two different situations can happen:
\begin{itemize}
\item If $\beta_n- \beta_j \not \in \mathbb{Z}$, we see that $I_{0^+}^{\beta_n- \beta_j} t^{\beta_n-k}$ will be always in the space $I_{0^+}^{\beta_n+(\beta_n-\beta_*+1)-\lceil \beta_n-\beta_* \rceil-\varepsilon}L^1[0,b]$ for every $\varepsilon>0$. This simply occurs because the worst choice is $\beta_j=\beta_*$ and $k=\lceil \beta_n -\beta_* \rceil$. Indeed, for $\varepsilon$ small enough, the previous space is contained in $I_{0^+}^{\beta_n}L^1[0,b]$, since $\beta_n-\beta_*+1-\lceil \beta_n-\beta_* \rceil$ is strictly positive.
\item If $\beta_n- \beta_j \in \mathbb{Z}^+$, there are two options:
\subitem If $\beta_n- \beta_j>\beta_n- \beta_*$, we have that $I_{0^+}^{\beta_n- \beta_j} t^{\beta_n-k}$ lies again in $I_{0^+}^{\beta_n}L^1[0,b]$, since the maximum value admitted for $k$ is $\lceil \beta_n -\beta_* \rceil$.
\subitem If $\beta_n- \beta_j<\beta_n- \beta_*$, we have that $I_{0^+}^{\beta_n- \beta_j} t^{\beta_n-k} \in \left \langle \left\{t^{\beta_n-k'}\right\} \right \rangle$ for some $k'<k$.
\end{itemize}
Thus, we can write $I_{0^+}^{\beta_n} w(t) - \Upsilon\, f(t)= I_{0^+}^{\beta_n} w_1(t) + f_1(t)$, and arrive to the equation $$\left(c_1 I_{0^+}^{\beta_n-\beta_1}+\dots+c_{n-1} I_{0^+}^{\beta_n-\beta_{n-1}}+ \textnormal{Id}\right) \left(u(t)-f(t)\right) = I_{0^+}^{\beta_n} w_1(t) + f_1(t).$$ Note that $f$ lived in a $\lceil \beta_n-\beta_* \rceil$ dimensional vector space, but $f_1$ lives in a (at most) $\lceil \beta_n-\beta_* \rceil-1$ dimensional vector space.

If we repeat the process, we obtain $$\left(c_1 I_{0^+}^{\beta_n-\beta_1}+\dots+c_{n-1} I_{0^+}^{\beta_n-\beta_{n-1}}+ \textnormal{Id}\right) \left(u(t)-f(t)-f_1(t)\right) = I_{0^+}^{\beta_n} w_2(t) + f_2(t),$$ with $f_2$ lying in a (at most) $\lceil \beta_n-\beta_* \rceil-2$ dimensional vector space. After enough iterations, the vector space has to be zero dimensional and we would have the situation $$\left(\Upsilon+ \textnormal{Id}\right) \left(u(t)-f(t)-\cdots-f_{r-1}(t)\right) = I_{0^+}^{\beta_n} w_r(t) \in I_{0^+}^{\beta_n} L^1[0,b].$$ Therefore, $u(t)-f(t)-\cdots-f_{r-1}(t) \in I_{0^+}^{\beta_n} L^1[0,b]$. Finally, if we use that $$f(t)+f_1(t)+\cdots+f_{r-1}(t) \in \left \langle \left\{t^{\beta_n-1} , \dots, t^{\beta_n-\lceil \beta_n-\beta_* \rceil} \right\} \right \rangle,$$ it follows $u \in \left \langle \left\{t^{\beta_n-1} , \dots, t^{\beta_n-\lceil \beta_n-\beta_* \rceil} \right\} \right \rangle \oplus I_{0^+}^{\beta_n}L^1[0,b]= \bigcap_{j=1}^n \mathcal{X}_{\beta_j}$.
\end{proof}

%We highlight that the first initial value in the previous remark can be, in fact, a fractional integral if $-1<\beta_n-\lceil \beta_n-\beta_* \rceil<0$.

\section{Smooth solutions for fractional differential equations}

Until this point we have checked that, a priori, there are more weak solutions (a $\lceil \beta_n \rceil$ dimensional space) than strong solutions (a $\lceil \beta_n-\beta_* \rceil$ dimensional space). We have also seen how weak solutions are codified depending on the source term, more concretely depending on the element chosen in $\ker D_{0^+}^{\beta_n}$. Moreover, we know that if the choice is made in a certain subspace of $\ker D_{0^+}^{\beta_n}$, then the obtained solution is a strong one. However, one could think about codifying strong solutions directly in the fractional differential equation via initial conditions, instead of using fractional integral problems and selecting a source term linked to a strong solution. Therefore, the last task should consist in relating the choices for $\ker D_{0^+}^{\beta_n}$ that give a strong solution with the corresponding initial conditions for the strong problem. 

First, to simplify the notation, we reconsider Equation (\ref{previodiff2}) with the additional hypotheses that each positive integer less or equal than $\beta_n-\beta_1$ can be written as $\beta_n-\beta_j$ for some $j$. 

\begin{equation}
D_{0^+}^{\beta_n} \, \left(c_1 I_{0^+}^{\beta_n-\beta_1}+\dots+c_{n-1} I_{0^+}^{\beta_n-\beta_{n-1}}+\textnormal{Id}\right) u(t) = w(t) \tag{\ref{previodiff2}}
\end{equation}

This does not imply a loss of generality, since we can assume that some $c_j=0$, if needed. The only purpose of this assumption is to ease the notation in this proof in the way that is described in the following paragraph. 

If $\beta_*=\beta_{n-m}$, then $\beta_n-\beta_{n-m}$ is the least possible non-integer difference $\beta_n-\beta_{j}$. Thus, we can use the previous notational assumption to check that $\beta_n-\beta_{n-j}=j$ for $j<m$ and $\beta_n-\beta_{n-m} \in (m-1,m)$. Thus, $\lceil \beta_n-\beta_{*} \rceil=\lceil \beta_n-\beta_{n-m} \rceil =m$ and $c_{n-m+1},\dots,c_{n-1}$ are $m-1$ constants multiplying integrals of integer order in (\ref{previodiff2}).

Now, we provide the main result of this section.

\begin{lemma} Under the previous notation, Equation (\ref{previodiff2}) with determined initial values $D_{0^+}^{\beta_n-m} u(0), \dots, D_{0^+}^{\beta_n-1} u(0)$ has a unique solution in $\bigcap_{j=1}^n \mathcal{X}_{\beta_j}$. This solution coincides with the unique solution of (\ref{cansado}), where the source term is the unique function $f \in \left \langle \left\{t^{\beta_n-1} , \dots, t^{\beta_n-m} \right\} \right \rangle$ fulfilling
\begin{align*}
D_{0^+}^{\beta_n-m}u(0)&=D_{0^+}^{\beta_n-m}f(0),\\
D_{0^+}^{\beta_n-m+1}u(0)+c_{n-1}\,D_{0^+}^{\beta_n-m}u(0)&=D_{0^+}^{\beta_n-m+1}f(0),\\
&\cdots \\
D_{0^+}^{\beta_n-1}u(0)+c_{n-1}\,D_{0^+}^{\beta_n-2}u(0)+\cdots+c_{n-m+1}&\,D_{0^+}^{\beta_n-m-1}u(0)=D_{0^+}^{\beta_n-1}f(0).
\end{align*}
\end{lemma}
\begin{proof}
Consider again the equation
\begin{equation}
\left(c_1 I_{0^+}^{\beta_n-\beta_1}+\dots+c_{n-1} I_{0^+}^{\beta_n-\beta_{n-1}}+ \textnormal{Id}\right) u(t) = I_{0^+}^{\beta_n} w(t) + f(t), \tag{\ref{cansado}}
\end{equation}

Recall that we look for strong solutions to (\ref{cansado}), that lie in the functional space $\left \langle \left\{t^{\beta_n-1} , \dots, t^{\beta_n-m} \right\} \right \rangle \oplus I_{0^+}^{\beta_n}L^1[0,b]$, so we write $$u(t)=d_1\,t^{\beta_n-m}+\cdots+d_{m}\,t^{\beta_n-1}+I_{0^+}^{\beta_n}\widetilde{u}(t).$$ Moreover, take into account that a strong choice for $f \in \ker D_{0^+}^{\beta_n}$ allows to describe $$f(t) = b_1\,t^{\beta_n-m}+\cdots+b_{m}\,t^{\beta_n-1}.$$ Now, we will derive the initial conditions after applying $D_{0^+}^{\beta_n-k}$, for every $k \in \{1,\dots,m\}$, and substituting $t=0$ in (\ref{cansado}).

At the right hand side this is easy, since $D_{0^+}^{\beta_n-k}I_{0^+}^{\beta_n} w(t) \in  I_{0^+}^1 L^1[0,b]$ and, thus, the substitution at $t=0$ gives zero. The function $D_{0^+}^{\beta_n-k}f(t)$ can be computed trivially, due to the expression of $f$, obtaining \[D_{0^+}^{\beta_n-k}f(0)=\Gamma(\beta_n - k +1) \,b_k.\]

At the left hand side, on the one hand, we have again a similar situation to the previous one, since $D_{0^+}^{\beta_n-k}I_{0^+}^{\beta_n-\beta_j}I_{0^+}^{\beta_n} \widetilde{u}(t) \in I_{0^+}^{k+(\beta_n-\beta_j)}L^1[0,b] \subset I_{0^+}^1 L^1[0,b]$ for any subindex $j\in \{1,\dots,n\}$ and, thus, the substitution at $t=0$ gives zero. On the other hand, $D_{0^+}^{\beta_n-k}I_{0^+}^{\beta_n-\beta_j}t^{\beta_n-l}$ has three possibilities:
\begin{itemize}
\item If $\beta_n-\beta_j>l-k$, then $D_{0^+}^{\beta_n-k}I_{0^+}^{\beta_n-\beta_j}t^{\beta_n-l}$ is a scalar multiple of a power of $t$ with positive exponent. Thus, when we make the substitution at $t=0$ we get $0$.
\item If $\beta_n-\beta_j=l-k$, then $D_{0^+}^{\beta_n-k}I_{0^+}^{\beta_n-\beta_j}t^{\beta_n-l}=\Gamma(\beta_n-l+1)$ is constant, and it is obviously defined for $t=0$.
\item If $\beta_n-\beta_j < l-k \leq m-1$, then $\beta_n-\beta_j$ is an integer and the computation $D_{0^+}^{\beta_n-k}I_{0^+}^{\beta_n-\beta_j}t^{\beta_n-l}$ gives the zero function.
\end{itemize}
The interest of the previous trichotomy is that we never obtain some $t^{-\gamma}$ with $\gamma>0$. In other case, we would have a huge trouble, since we could not evaluate the expression for $t=0$. Fortunately, we can always apply $D_{0^+}^{\beta_n-k}$ to Equation (\ref{cansado}), for every value $k \in \{1,\dots,m\}$, and substitute at $t=0$. We arrive to the following linear system of equations
\begin{align*}
D_{0^+}^{\beta_n-m}u(0)&=D_{0^+}^{\beta_n-m}f(0),\\
D_{0^+}^{\beta_n-m+1}u(0)+c_{n-1}\,D_{0^+}^{\beta_n-m}u(0)&=D_{0^+}^{\beta_n-m+1}f(0),\\
&\cdots \\
D_{0^+}^{\beta_n-1}u(0)+c_{n-1}\,D_{0^+}^{\beta_n-2}u(0)+\cdots+c_{n-m+1}&\,D_{0^+}^{\beta_n-m}u(0)=D_{0^+}^{\beta_n-1}f(0).
\end{align*}
Note that all the involved derivatives in the initial conditions have the same decimal part, since only coefficients $c_{n-1},\cdots,c_{n-m+1}$ appear in the system. We also highlight that the system has always a unique solution, since it is triangular and it has no zero element in the diagonal. Therefore, a choice for $f$ linked to a strong solution determines a vector of initial values $(D_{0^+}^{\beta_n-m}u(0),\dots,D_{0^+}^{\beta_n-1}u(0))$ and vice-versa in a bijective way.
\end{proof}

We shall give two examples summarizing how to apply all the previous results.

\begin{example}
Consider the following fractional differential equation (strong problem)
\begin{equation*}
\left(D_{0^+}^{\frac{7}{3}}+3\,D_{0^+}^{\frac{4}{3}}+4\,D_{0^+}^{\frac{1}{3}}\right)u(t)=t^3
\end{equation*}
and define $\beta_1=\frac{1}{3},\beta_2=\frac{4}{3},\beta_3=\frac{7}{3}$. In this case, note that $\beta_*=0$, since all the differences $\beta_3-\beta_j$ are integers. The strong solutions for the example will lie in $\bigcap_{j=1}^3 \mathcal{X}_{\beta_j}$. The dimension of the affine space of strong solutions will be $\lceil \beta_3\rceil=3$ and the initial conditions that ensure existence and uniqueness of solution will be $D_{0^+}^{\frac{4}{3}}u(0)=a_3$, $D_{0^+}^{\frac{1}{3}}u(0)=a_2$ and $I_{0^+}^{\frac{2}{3}}u(0)=a_1$.

Moreover, after left-factoring $D_{0^+}^{\frac{7}{3}}$, we find that the associated family of weak problems is
\begin{equation*}
\left(4\,I_{0^+}^{2}+3\,I_{0^+}^1+\textnormal{Id}\right)u(t)=I_{0^+}^{\frac{7}{3}} t^3 + f(t)
\end{equation*}
where $f(t) \in \left \langle \left\{t^{\frac{4}{3}}, t^{\frac{1}{3}}, t^{-\frac{2}{3}} \right\} \right \rangle$, which lives in a three dimensional space. The, a priori weak, obtained solution is always strong since we have that $\lceil \beta_3 - \beta_*\rceil =\lceil \beta_3 \rceil$.

Finally, the relation between a choice for $f(t)=b_3\,t^{\frac{4}{3}}+b_2\,t^{\frac{1}{3}}+b_1\,t^{-\frac{2}{3}}$ providing a strong solution and the initial conditions $a_1$, $a_2$ and $a_3$ is
\begin{align*}
a_1=I_{0^+}^{\frac{2}{3}}f(0)=b_1 \cdot \Gamma\left(1-\frac{2}{3}\right), \\
a_2+3\,a_1=D_{0^+}^{\frac{1}{3}}f(0)=b_2 \cdot \Gamma\left(1+\frac{1}{3}\right),\\
a_3+3\,a_2+4\,a_3=D_{0^+}^{\frac{4}{3}}f(0)=b_3 \cdot \Gamma\left(1+\frac{4}{3}\right).
\end{align*}
\end{example}

\begin{example}
Consider the following fractional differential equation (strong problem)
\begin{equation*}
\left(D_{0^+}^{\frac{13}{4}}+3\,D_{0^+}^{\frac{9}{4}}+D_{0^+}^{2}+D_{0^+}^{\frac{5}{4}}+D_{0^+}^{1}\right)u(t)=t
\end{equation*}
and define $\beta_1=1,\beta_2=\frac{5}{4},\beta_3=2,\beta_4=\frac{9}{4},\beta_5=\frac{13}{4}$. In this case, note that $\beta_*=\beta_3$, since it fulfils the property that $\beta_5-\beta_*$ is the least possible non-integer difference $\beta_5-\beta_j$. The strong solutions for the example will lie in $\bigcap_{j=1}^5 \mathcal{X}_{\beta_j}$. The dimension of the affine space of strong solutions will be $\lceil \beta_5-\beta_* \rceil=2$ and the initial conditions that ensure existence and uniqueness of solution will be $D_{0^+}^{\frac{9}{4}}u(0)=a_2$ and $D_{0^+}^{\frac{5}{4}}u(0)=a_1$.

Moreover, after left-factoring $D_{0^+}^{\frac{13}{4}}$, we find that the associated family of weak problems is
\begin{equation*}
\left(I_{0^+}^{\frac{9}{4}}+I_{0^+}^{2}+I_{0^+}^{\frac{5}{4}}+3\,I_{0^+}^{1}+\textnormal{Id}\right)u(t)=I_{0^+}^{\frac{13}{4}} t + f(t)
\end{equation*}
where $f(t) \in \left \langle \left\{ t^{\frac{9}{4}}, t^{\frac{5}{4}}, t^{\frac{1}{4}}, t^{-\frac{3}{4}} \right\} \right \rangle$, which lives in a four dimensional space. The, a priori weak, obtained solution will be strong if $f(t) \in \left \langle \left\{ t^{\frac{9}{4}}, t^{\frac{5}{4}} \right\} \right \rangle$.

Finally, the relation between a choice for $f(t)=b_2\,t^{\frac{9}{4}}+b_1\,t^{\frac{5}{4}}$ providing a strong solution and the initial conditions $a_1$ and $a_2$ is
\begin{align*}
a_1=D_{0^+}^{\frac{5}{4}}u(0)=D_{0^+}^{\frac{5}{4}}f(0)=b_1 \cdot \Gamma\left(1+\frac{5}{4}\right) \\
a_2+3\,a_1=D_{0^+}^{\frac{9}{4}}u(0)+3\,D_{0^+}^{\frac{5}{4}}u(0)=D_{0^+}^{\frac{9}{4}}f(0)=b_2 \cdot \Gamma\left(1+\frac{9}{4}\right)
\end{align*}
\end{example}

\section{Conclusions}\label{s:conc}

We summarize the conclusions obtained in this paper.

\begin{itemize}
\item We have recalled the main results involving existence and uniqueness of solution for linear fractional integral equations with constant coefficients.
\item We have seen that from each linear fractional differential equation with constant coefficients of order $\beta_n$ is possible to derive a $\lceil \beta_n \rceil$ dimensional family of associated fractional integral equations, in a natural way. Moreover, each solution to the fractional differential equation fulfils exactly one of these fractional integral equations.
\item We have shown that there exists a $\lceil \beta_n-\beta_*\rceil$ dimensional subfamily (of the $\lceil \beta_n \rceil$ dimensional family of associated fractional integral equations) such that each solution to a problem of the subfamily gives a solution to the original linear fractional differential equation of order $\beta_n$.  This value $\beta_*$ is obtained as the greatest involved order in the fractional differential equation such that $\beta_n-\beta_*$ is not an integer. If such a value does not exist, the same result holds after defining $\beta_*=0$.
\item We have seen how initial values at $t=0$ for the derivatives of orders $\beta_n-\lceil \beta_n-\beta_* \rceil,\dots,\beta_n-1$ guarantee existence and uniqueness of solution to a linear fractional differential equation with constant coefficients of order $\beta_n$. We have described the correspondence between such initial values for the fractional differential equation and the selection of a source term in the $\lceil \beta_n-\beta_*\rceil$ dimensional subfamily, in such a way that both problems have the same unique solution. If $\beta_n-\lceil \beta_n-\beta_* \rceil \in (-1,0)$ this first initial value is imposed, indeed, for the fractional integral of order $\lceil \beta_n-\beta_* \rceil-\beta_n$.
\item We expect that this idea can be extended to different type of fractional differential problems. It would be nice to amplify the scope of this work to a more general case than the one of constant coefficients. Moreover, the same idea could be applied to obtain similar results for other derivatives like Caputo and compare them with the already existing theory in literature and solution methods \cite{AsCaTu}. Furthermore, the same philosophy could be applied to other type of problems as, for instance, periodic ones that have relevant applications \cite{St}.
\end{itemize}

\section*{Acknowledgements}
The research of D. Cao Labora was partially supported by a PhD scholarship from Ministerio de Educaci\'on, Cultura y Deporte (FPU) with reference number FPU16/04168.

\bibliographystyle{amsplain}

\begin{thebibliography}{99}
\normalsize

\bibitem{AsCaTu} R. Ashurov, A. Cabada, B. Turmetov, Operator method for construction of solutions of linear fractional differential equations with constant coefficients. \emph{Fractional Calculus and Applied Analysis} \textbf{19}, No 1 (2016), 229--252; DOI: 10.1515/fca-2016-0013.

\bibitem{CaRo} D. Cao Labora and R. Rodr\'iguez-L\'opez, From fractional order equations to integer order equations. \emph{Fractional Calculus and Applied Analysis} \textbf{20}, No 6 (2017), 1405--1423; DOI: 10.1515/fca-2017-0074.

\bibitem{CaMc} D. I. Cartwright and J. R. McMullen, A note on the fractional calculus. \emph{Proceedings of the Edinburgh Mathematical Society} \textbf{21}, No 1 (1978), 79--80; DOI: 10.1017/S0013091500015911.

\bibitem{DiFo} K. Diethelm and N.J. Ford, Numerical solution of the Bagley-Torvik equation. \emph{BIT} \textbf{42}, No 3 (2002), 490--507; DOI: 10.1023/A: 1021973025166.

\bibitem{Lu} R. Hilfer and Y. Luchko, Desiderata for Fractional Derivatives and Integrals. \emph{Mathematics} \textbf{7}, No 2 (2019), 149--153; DOI: 10.3390/math7020149.

\bibitem{Rust} W. M. Rust, A theorem on Volterra integral equations of the second kind with discontinuous kernels. \emph{American Mathematical Monthly} \textbf{41}, No 6 (1934), 346--350; DOI: 10.2307/2301549.

\bibitem{St} S. Stan\v{e}k, Periodic problem for the generalized Basset fractional differential equation. \emph{Fractional Calculus and Applied Analysis} \textbf{18}, No 5 (2015), 1277--1290; DOI: 10.1515/fca-2015-0073.

\bibitem{MMKA}J.A. Tenreiro Machado, F. Mainardi, V. Kiryakova, T. Atanackovi\'{c}, Fractional Calculus: D'o\`u venons-nous? Que sommes-nous? O\`u Allons-nous? \emph{Fractional Calculus and Applied Analysis} \textbf{19}, No 5 (2016),  1074--1104; DOI: 10.1515/fca-2016-0059.

\bibitem{Titchmarsh} E. C. Titchmarsh, The zeros of certain integral functions. \emph{Proceedings of the London Mathematical Society} \textbf{25}, No 2 (1926), 283--302; DOI: 10.1112/plms/s2-25.1.283.

\bibitem{Maina} A. Carpinteri and F. Mainardi, \emph{Fractals and Fractional Calculus in Continuum Mechanics}, 223--276. Springer Verlag, Wien and New York (1997).

\bibitem{KiSrTr} A. Kilbas, H.M. Srivastava and J.J. Trujillo, \emph{Theory and Applications of Fractional Differential Equations}. Elsevier, Amsterdam (2006).

\bibitem{MiRo} K. S. Miller, B. Ross, \emph{An Introduction to the Fractional Calculus and Fractional Differential Equations}. John Wiley \& Sons, United States of America (1993).

\bibitem{Pod} I. Podlubny,  \emph{Fractional Differential Equations}. Academic Press, San Diego (1999).

\bibitem{Samko} S. Samko, A. Kilbas, and O. Marichev, \emph{Fractional Integrals and Derivatives. Theory and Applications}, Gordon and Breach, Yverdon (1993).
\end{thebibliography}

\end{document}